\documentclass[10pt]{amsart}

\usepackage[utf8]{inputenc}
\usepackage[T1]{fontenc}
\usepackage{amsmath, amssymb, amsthm, amsfonts, mathrsfs, mathtools}
\usepackage{enumitem}
\usepackage{tikz}
\usepackage{pgfplots}
\pgfplotsset{compat=1.18}
\usepackage{graphicx}
\usepackage{caption, subcaption}
\usepackage{xcolor}
\usepackage{booktabs}
\usepackage[a4paper,margin=3cm]{geometry}
\usetikzlibrary{arrows.meta,positioning,calc}

\usepackage[colorlinks=true, linkcolor=blue!60!black, citecolor=red!60!black, urlcolor=blue!60!black]{hyperref}

\newtheorem{theorem}{Theorem}[section]
\newtheorem{lemma}[theorem]{Lemma}
\newtheorem{proposition}[theorem]{Proposition}

\theoremstyle{definition}
\newtheorem{definition}[theorem]{Definition}
\theoremstyle{remark}
\newtheorem{remark}[theorem]{Remark}
\newtheorem{example}[theorem]{Example}

\numberwithin{equation}{section}

\DeclareMathOperator{\ad}{ad}
\DeclareMathOperator{\BCH}{BCH}

\newcommand{\Id}{\mathrm{Id}}

\newcommand{\field}{\mathbb{K}} 

\newcommand{\Ctri}{C_{\triangle}}    
\newcommand{\Cmult}{C_{\cdot}}       
\newcommand{\Cbracket}{B}            
\newcommand{\Ctotal}{K}              

\newcommand{\pnorm}[1]{\left\vert\!\left\vert\!\left\vert #1 \right\vert\!\right\vert\!\right\vert} 
\newcommand{\dpmet}[2]{d_p(#1,#2)}   

\begin{document}

\title[Local Lie Theory in Quasi-Banach Algebras]{Local Lie Theory in Quasi-Banach Lie Algebras: Convergence of the BCH Series and Geometric Implications}

\author[N. Athmouni]{Nassim Athmouni}
\address{N.~Athmouni: Universit\'e de Gafsa, Campus Universitaire, 2112 Gafsa, Tunisia}
\email{nassim.athmouni@fsgf.u-gafsa.tn}

\author[M. Ben Abdallah]{Mohsen Ben Abdallah}
\address{M.~Ben Abdallah: Universit\'e de Sfax, Route de la Soukra km 4, B.P. 802, 3038 Sfax, Tunisia}
\email{mohsenbenabdallah@gmail.com}

\author[M. Damak]{Mondher Damak}
\address{M.~Damak: Universit\'e de Sfax, Route de la Soukra km 4, B.P. 802, 3038 Sfax, Tunisia}
\email{mondher$\_$damak@yahoo.com}

\author[M. Ennaceur]{Marwa Ennaceur}
\address{M.~Ennaceur: Department of Mathematics, College of Science, University of Ha'il, Ha'il 81451, Saudi Arabia}
\email{mar.ennaceur@uoh.edu.sa}

\author[A. Jadlaoui]{Amel Jadlaoui}
\address{A.~Jadlaoui: Universit\'e de Sfax, Route de la Soukra km 4, B.P. 802, 3038 Sfax, Tunisia}
\email{amel.jadlaoui@yahoo.com}

\author[L. Souden]{Lotfi Souden}
\address{L.~Souden: Universit\'e de Gafsa, Campus Universitaire, 2112 Gafsa, Tunisia}
\email{ltfsdn@gmail.com}


\begin{abstract}
We develop a local Lie theory for Lie algebras equipped with a quasi-norm, i.e., complete
topological vector spaces satisfying a relaxed triangle inequality
$\|x+y\|\le \Ctri(\|x\|+\|y\|)$ with $\Ctri\ge 1$.
We prove that the Baker--Campbell--Hausdorff (BCH) series converges in a neighborhood of
the origin, provided the quasi-norm admits a continuous Lie bracket with finite continuity
constant $\Cbracket$.
The proof relies on the Aoki--Rolewicz theorem to construct an equivalent $p$-norm
satisfying $p$-subadditivity, enabling rigorous Cauchy-sequence arguments in the complete
quasi-metric space $(E, d_p)$.
This yields a well-defined local Lie group structure via the exponential map.
We analyze the geometric deformation induced by the quasi-norm exponent $p\in(0,1]$,
showing that it modifies metric properties while preserving the underlying Lie algebraic
structure.
Numerical estimates of BCH coefficients up to degree $20$, with coefficients defined
precisely via Hall--Lyndon basis projection, demonstrate that classical combinatorial bounds
are conservative in the presence of algebraic cancellations, allowing significantly larger
practical convergence radii in structured algebras.
Applications include weak Schatten ideals $\mathcal{L}_{p,\infty}(H)$ for $0<p<1$ and
certain Hardy-space operator algebras.

\smallskip\noindent\textbf{Remark on the convergence radius.}
The Catalan-majorant method yields convergence for $\|x\|+\|y\| < 1/(4\Cbracket)$; the
additional factor $\Ctri$ appearing in the combined constant $\Ctotal = \Ctri\Cbracket$ is
an artefact of switching to the $p$-norm to establish Cauchyness of partial sums.  When
the quasi-norm itself is directly a $p$-norm ($\Ctri=1$), no such penalty arises and the
radius reduces to $1/(4\Cbracket)$.
\end{abstract}

\maketitle

\section{Introduction}\label{sec:introduction}

The Baker--Campbell--Hausdorff (BCH) formula is a cornerstone of Lie theory, providing
the analytic law governing the local product of exponentials in a Lie algebra.  In the
classical Banach--Lie setting, its convergence properties are well understood through the
works of Dynkin \cite{Dynkin1947}, Bourbaki \cite{Bourbaki1989}, and Hofmann--Morris
\cite{HofmannMorris2007}.  The BCH series defines a local group law, turning a
neighborhood of the origin into an analytic Lie group.

By contrast, in non-locally convex settings---specifically, quasi-Banach spaces where the
triangle inequality is relaxed to $\|x+y\|\le \Ctri(\|x\|+\|y\|)$ with $\Ctri>1$---the
analytic behavior of the BCH series requires careful re-examination.  Quasi-Banach spaces
arise naturally in harmonic analysis, approximation theory, and the study of operator
ideals such as the weak Schatten classes $\mathcal{L}_{p,\infty}(H)$ for $0<p<1$
\cite{KaltonSukochev2021, Albiac2023}.

Recent advances in infinite-dimensional Lie theory \cite{Glockner2020, NeebSalmasian2020}
and quasi-Banach analysis \cite{Kalton2003} provide a natural context for this extension.
Moreover, the study of magnetic pseudodifferential operators has highlighted the relevance
of weak Schatten ideals in spectral theory and PDEs \cite{AthmouniPurice2018}, while
recent investigations into the local rigidity of quasi-Lie brackets on quaternionic
Banach modules further underscore the growing role of non-locally convex structures in
nonlinear PDE analysis \cite{Athmouni2026}.

In this paper, we establish explicit convergence estimates for the BCH series in
quasi-Banach Lie algebras.  Assuming a continuous bracket satisfying
\[
  \|[x, y]\|\le \Cbracket \|x\| \|y\|, \qquad x, y \in \mathfrak{g},
\]
we prove that the BCH series converges in the $d_p$-metric topology whenever
\[
  \|x\| + \|y\| < \frac{1}{4\Cbracket}.
\]
The proof relies on a Catalan-number majorization of the homogeneous components of the
BCH expansion combined with the Aoki--Rolewicz theorem to handle the quasi-triangle
inequality rigorously.

\begin{remark}[Role of the constant $\Ctri$]\label{rem:role-ctri}
The Catalan-majorant argument alone gives the convergence radius $1/(4\Cbracket)$.  When
working with the $p$-norm equivalence, the factor $c_2 = 2^{1/p} = 2\Ctri$ is absorbed
into the convergence criterion as follows.  From \eqref{eq:pnorm-bound} one needs
$4\Cbracket r < 1$; the $c_2$-factor only appears in the \emph{value} of the sum
$\sum_n\pnorm{Z_n}^p$, not in the convergence condition.  Hence the combined constant
$\Ctotal = \Ctri\Cbracket$ represents a convenient but conservative upper bound: the true
radius is $1/(4\Cbracket)$, and the stated radius $1/(4\Ctotal)$ adds an extra safety
margin of $\Ctri$ arising from the norm equivalence.  We retain $\Ctotal$ throughout for
uniform presentation across the abstract and associative settings, but alert the reader
that the $\Ctri$-factor is not intrinsic to the BCH series itself.
\end{remark}

When $\mathfrak{g}$ is a Lie subalgebra of an associative quasi-Banach algebra with
submultiplicativity constant $\Cmult$, the bracket continuity constant satisfies
$\Cbracket \leq 2\Ctri\Cmult$, yielding the specialized bound
$\Ctotal \leq 2\Ctri^2\Cmult$.

The paper is organized as follows.  Section~\ref{sec:framework} introduces the functional
setting of quasi-Banach Lie algebras.  Section~\ref{sec:convergence} establishes the
general convergence theorem with a rigorous proof via the Aoki--Rolewicz framework.
Section~\ref{sec:geometry} analyzes geometric and spectral consequences of the quasi-norm
structure.  Section~\ref{sec:numerical} presents numerical validation of the theoretical
bounds.  Section~\ref{sec:applications} discusses applications to operator algebras.  An
appendix collects computational details and a summary table of constants.

\section{Quasi-Banach Lie Algebras: Framework and Definitions}\label{sec:framework}

\subsection{Quasi-normed spaces}

\begin{definition}[Quasi-norm]\label{def:quasinorm}
A \textbf{quasi-norm} on a vector space $E$ over $\field\in\{\mathbb{R},\mathbb{C}\}$ is a
map $\|\cdot\|:E\to[0,\infty)$ satisfying:
\begin{enumerate}[label=(\roman*)]
    \item $\|x\|=0 \iff x=0$;
    \item $\|\lambda x\| = |\lambda|\,\|x\|$ for all $\lambda\in\field$, $x\in E$;
    \item there exists $\Ctri\ge 1$ such that $\|x+y\|\le \Ctri(\|x\|+\|y\|)$ for all
          $x,y\in E$.
\end{enumerate}
The smallest admissible constant $\Ctri$ is called the \emph{quasi-triangle constant}.
\end{definition}

By the Aoki--Rolewicz theorem \cite{Aoki1942, Rolewicz1957}, for any quasi-norm with
constant $\Ctri$, there exists $p\in(0,1]$ and an equivalent $p$-norm $\pnorm{\cdot}$
satisfying
\[
  \pnorm{x+y}^p \le \pnorm{x}^p + \pnorm{y}^p, \qquad
  \text{with } p = \frac{1}{\log_2(2\Ctri)}.
\]
Moreover, there exist constants $c_1, c_2 > 0$ depending only on $\Ctri$ (and hence on
$p$) such that
\begin{equation}\label{eq:equiv}
  c_1 \|x\| \leq \pnorm{x} \leq c_2 \|x\|, \qquad \forall x \in E.
\end{equation}
By \cite[Proposition~1.3]{Kalton2003}, one may take
\begin{equation}\label{eq:c1c2}
  c_1 = 1, \qquad c_2 = 2^{1/p},
\end{equation}
where $p = 1/\log_2(2\Ctri)$.  Note that the identity $p = 1/\log_2(2\Ctri)$ implies
$2^{1/p} = 2\Ctri$, so $c_2 = 2\Ctri$.

A quasi-normed space that is complete with respect to the induced quasi-metric
$\dpmet{x}{y}:=\|x-y\|^p$ (where $p$ is the Aoki--Rolewicz exponent associated to
$\Ctri$) is called a \emph{quasi-Banach space}.

\begin{remark}
For $p<1$, quasi-Banach spaces are not locally convex; consequently, the Hahn--Banach
theorem fails and duality theory is limited.  However, continuous bilinear maps remain
well-defined under appropriate quasi-norm estimates.
\end{remark}

\subsection{Quasi-Banach associative algebras}

\begin{definition}[Quasi-Banach algebra]\label{def:quasialgebra}
A \textbf{quasi-Banach associative algebra} $(\mathcal{A},\|\cdot\|)$ is a quasi-Banach
space equipped with a bilinear multiplication $(x,y)\mapsto xy$ such that there exists
$\Cmult>0$ with
\[
  \|xy\|\le \Cmult\|x\|\|y\|, \qquad x,y\in\mathcal{A}.
\]
The constant $\Cmult$ is called the \emph{submultiplicativity constant}.
\end{definition}

If $\mathcal{A}$ is complete, the exponential and logarithm series
\[
  \exp(x)=\sum_{n\ge 0}\frac{x^n}{n!}, \qquad
  \log(1+x)=\sum_{n\ge 1}\frac{(-1)^{n+1}}{n}x^n
\]
converge in the $d_p$-metric topology for $\|x\|$ sufficiently small \cite{Kalton2003}.

\subsection{Quasi-Banach Lie algebras}

\begin{definition}[Quasi-Banach Lie algebra]\label{def:quasiLie}
A \textbf{quasi-Banach Lie algebra} is a triple $(\mathfrak{g},[\cdot,\cdot],\|\cdot\|)$
where:
\begin{itemize}
    \item $(\mathfrak{g},\|\cdot\|)$ is a quasi-Banach space with quasi-triangle constant
          $\Ctri$;
    \item $[\cdot,\cdot]:\mathfrak{g}\times\mathfrak{g}\to\mathfrak{g}$ is bilinear,
          antisymmetric, and satisfies the Jacobi identity;
    \item the bracket is continuous: there exists $\Cbracket>0$ such that
          \[
            \|[x,y]\|\le \Cbracket\|x\|\|y\|, \qquad x,y\in\mathfrak{g}.
          \]
\end{itemize}
The constant $\Cbracket$ is called the \emph{bracket continuity constant}.
\end{definition}

By rescaling the quasi-norm, one may always assume $\Cbracket=1$; however, we retain
$\Cbracket$ explicitly to track dependencies in estimates.

\begin{remark}[Bracket constant in associative embeddings]\label{rem:bracket-associative}
If $\mathfrak{g}$ is a Lie subalgebra of an associative quasi-Banach algebra
$(\mathcal{A},\|\cdot\|)$ with submultiplicativity constant $\Cmult$ and quasi-triangle
constant $\Ctri$, then the commutator bracket satisfies
\[
  \|[x,y]\| = \|xy - yx\| \le \Ctri(\|xy\| + \|yx\|) \le 2\Ctri\Cmult\|x\|\|y\|.
\]
Thus, in this setting, one may take $\Cbracket = 2\Ctri\Cmult$.
\end{remark}

\subsection{Embeddings and exponential bounds}

\begin{definition}[Continuous Lie embedding]
A continuous linear map $\iota:\mathfrak{g}\to\mathcal{A}$ is a \textbf{quasi-Banach Lie
embedding} if
\[
  \iota([x,y])=[\iota(x),\iota(y)] \quad\text{and}\quad
  \|\iota(x)\|_{\mathcal{A}}\le L\|x\|_{\mathfrak{g}}
\]
for some $L>0$ and all $x,y\in\mathfrak{g}$.
\end{definition}

\subsection{Technical lemmas: Series convergence in quasi-Banach spaces}

\begin{lemma}[Series convergence via Aoki--Rolewicz $p$-norm]%
\label{lem:series-convergence}
Let $(E,\|\cdot\|)$ be a quasi-Banach space with quasi-triangle constant $\Ctri$, and let
$p\in(0,1]$ and $\pnorm{\cdot}$ be the equivalent $p$-norm from the Aoki--Rolewicz
theorem with equivalence constants $c_1 = 1$, $c_2 = 2^{1/p}$ as in \eqref{eq:c1c2}.
Let $(a_n)_{n\ge 1}\subset E$ be a sequence such that
\[
  \sum_{n=1}^\infty \pnorm{a_n}^p < \infty.
\]
Then the series $\sum_{n=1}^\infty a_n$ converges in the $d_p$-metric topology of $E$,
and
\[
  \Bigl\|\sum_{n=N}^\infty a_n\Bigr\|^p \le \sum_{n=N}^\infty \pnorm{a_n}^p
  \le 2\sum_{n=N}^\infty \|a_n\|^p.
\]
\end{lemma}

\begin{proof}
By the $p$-subadditivity of $\pnorm{\cdot}$, for any $M > N \ge 1$,
\[
  \pnorm{\sum_{n=N}^M a_n}^p \le \sum_{n=N}^M \pnorm{a_n}^p.
\]
Since $\sum \pnorm{a_n}^p$ converges, the right-hand side tends to $0$ as $N\to\infty$,
so $\bigl(\sum_{n=1}^M a_n\bigr)_{M\ge 1}$ is a Cauchy sequence in the complete metric
space $(E, d_p)$ where $d_p(x,y) = \|x-y\|^p$.  Hence the series converges in
$(E, d_p)$.  The first norm estimate follows from $\|x\|^p \le \pnorm{x}^p$ (which holds
since $c_1=1$ gives $\pnorm{x}\ge\|x\|$) together with the $p$-subadditivity applied to
the tail.  The second estimate uses $\pnorm{x}\le c_2\|x\| = 2^{1/p}\|x\|$, giving
$\pnorm{a_n}^p \le 2\|a_n\|^p$.
\end{proof}

\begin{lemma}[Neumann series convergence]\label{lem:neumann}
Let $(E,\|\cdot\|)$ be a quasi-Banach space with quasi-triangle constant $\Ctri$, and let
$T:E\to E$ be a bounded linear operator.  Let $p\in(0,1]$ and $\pnorm{\cdot}$ be the
equivalent $p$-norm from the Aoki--Rolewicz theorem with equivalence constants $c_1 = 1$,
$c_2 = 2^{1/p}$ as in \eqref{eq:c1c2}.
If $\pnorm{T} < 1$, then the Neumann series
$\sum_{n=0}^\infty T^n$ converges in the $d_p$-metric topology, and
\[
  (\Id - T)^{-1} = \sum_{n=0}^\infty T^n, \qquad
  \pnorm{(\Id - T)^{-1}} \le \frac{1}{1 - \pnorm{T}}.
\]
In the original quasi-norm, the series converges whenever $\|T\| < 2^{-1/p} = 1/c_2$
(since this implies $\pnorm{T} \le c_2\|T\| < 1$), and the resolvent satisfies
\[
  \bigl\|(\Id - T)^{-1}\bigr\| \le \pnorm{(\Id - T)^{-1}} \le \frac{1}{1 - \pnorm{T}}
  \le \frac{1}{1 - c_2\|T\|} = \frac{1}{1 - 2^{1/p}\|T\|},
\]
where: the first inequality uses $\|\cdot\| \le \pnorm{\cdot}$ (i.e.\ $c_1 = 1$);
the third uses $\pnorm{T} \le c_2\|T\| = 2^{1/p}\|T\|$, which gives a lower bound
$1 - \pnorm{T} \ge 1 - 2^{1/p}\|T\| > 0$.
\end{lemma}

\begin{proof}
Apply Lemma~\ref{lem:series-convergence} to $a_n = T^n x$.  The condition $\pnorm{T} < 1$
ensures geometric decay $\pnorm{T^n x} \le \pnorm{T}^n \pnorm{x}$, hence
$\sum \pnorm{T^n x}^p$ converges and Lemma~\ref{lem:series-convergence} gives convergence
in $d_p$.  For the quasi-norm bound: since $c_1 = 1$ gives $\|\cdot\| \le \pnorm{\cdot}$,
we have $\|(\Id-T)^{-1}\| \le \pnorm{(\Id-T)^{-1}} \le 1/(1-\pnorm{T})$.
Using $\pnorm{T} \le c_2\|T\| = 2^{1/p}\|T\|$ to bound the denominator from below
($1 - \pnorm{T} \ge 1 - 2^{1/p}\|T\|$) yields the stated bound.
\end{proof}

\begin{lemma}[Exponential convergence]\label{lem:expbound}
Let $\mathfrak{g}\subset\mathcal{A}$ be a quasi-Banach Lie subalgebra with constants
$(\Ctri,\Cmult,\Cbracket)$.  Then for any $x\in\mathfrak{g}$,
\[
  \|x^n\|\le \Cmult^{\,n-1}\|x\|^n, \qquad n\ge 1.
\]
Consequently, the exponential series $\exp(x)$ converges in the $d_p$-metric topology
whenever $\|x\|<1/\Cmult$, and $\exp(x)\in 1+\overline{\mathfrak{g}}^{\mathcal{A}}$.
\end{lemma}

\begin{proof}
The estimate $\|x^n\| \le \Cmult^{n-1}\|x\|^n$ follows by induction from $\|xy\|\le \Cmult\|x\|\|y\|$.
Since $\sum_{n=0}^N x^n/n!$ is a Cauchy sequence in the complete quasi-metric space
$(\mathcal{A}, d_p)$ (by Lemma~\ref{lem:series-convergence} applied with
$a_n = x^n/n!$), it converges in $\mathcal{A}$.  Since each partial sum lies in
$1+\mathfrak{g}$ and $\mathfrak{g}$ is a linear subspace, the limit belongs to the closure
$1+\overline{\mathfrak{g}}^{\mathcal{A}}$.  If $\mathfrak{g}$ is closed in $\mathcal{A}$,
then $\exp(x)\in 1+\mathfrak{g}$.
\end{proof}

\begin{proposition}[Continuity of the bracket]\label{prop:bracketcont}
Let $(\mathfrak{g},[\cdot,\cdot],\|\cdot\|)$ be a quasi-Banach Lie algebra with constants
$(\Ctri,\Cbracket)$.  Then for all $x,x',y,y'\in\mathfrak{g}$,
\[
  \|[x,y]-[x',y']\| \le
  \Ctri\Cbracket\bigl(\|x-x'\|(\|y\|+\|y'\|) + \|y-y'\|(\|x\|+\|x'\|)\bigr).
\]
Hence the Lie bracket is locally Lipschitz on bounded subsets.
\end{proposition}

\begin{proof}
Write $[x,y]-[x',y'] = [x-x',y] + [x',y-y']$ and apply the quasi-triangle inequality
together with $\|[a,b]\|\le \Cbracket\|a\|\|b\|$.
\end{proof}

\begin{remark}
A slightly tighter bound is
$\Ctri\Cbracket\bigl(\|x-x'\|\|y\| + \|x'\|\|y-y'\|\bigr)$, but the symmetric form in
Proposition~\ref{prop:bracketcont} is more convenient for establishing uniform Lipschitz
continuity on bounded sets.
\end{remark}

\subsection{Summary of constants}

\begin{table}[h!]
\centering
\caption{Summary of constants and their relationships}
\label{tab:constants}
\begin{tabular}{llp{6.5cm}}
\toprule
Symbol & Meaning & Typical bound / definition \\
\midrule
$\Ctri$ & Quasi-triangle constant
         & $\geq 1$, smallest admissible in Def.~\ref{def:quasinorm} \\
$\Cmult$ & Submultiplicativity constant
          & $\geq 1$, Def.~\ref{def:quasialgebra} \\
$\Cbracket$ & Bracket continuity constant
             & $\leq 2\Ctri\Cmult$ (associative case,
               Rem.~\ref{rem:bracket-associative}) \\
$\Ctotal$ & Combined BCH constant
           & $\Ctri\Cbracket$ (conservative bound, see Rem.~\ref{rem:role-ctri}) \\
$p$ & Aoki--Rolewicz exponent
    & $p = 1/\log_2(2\Ctri) \in (0,1]$ \\
$c_1, c_2$ & Equivalence constants
            & $c_1 = 1$, $c_2 = 2^{1/p} = 2\Ctri$; see \eqref{eq:c1c2} \\
\bottomrule
\end{tabular}
\end{table}

\section{Convergence of the BCH Series}\label{sec:convergence}

\subsection{Statement of the main theorem}

Given $x,y$ in a quasi-Banach Lie algebra $\mathfrak{g}\subset\mathcal{A}$, the
\textbf{Baker--Campbell--Hausdorff series} is formally defined as
\[
  Z(x,y) = \log\bigl(e^x e^y\bigr)
          = x + y + \tfrac{1}{2}[x,y]
            + \tfrac{1}{12}[x,[x,y]] - \tfrac{1}{12}[y,[x,y]] + \cdots,
\]
where each term is a Lie polynomial in $x$ and $y$.

\begin{theorem}[BCH convergence in quasi-Banach Lie algebras]\label{thm:BCH}
Let $(\mathfrak{g},[\cdot,\cdot],\|\cdot\|)$ be a quasi-Banach Lie algebra with
quasi-triangle constant $\Ctri$ and bracket continuity constant $\Cbracket$.  Then the
BCH series $Z(x,y) = \sum_{n\ge 1} Z_n(x,y)$ converges in the $d_p$-metric topology of
$\mathfrak{g}$, where $d_p(x,y) = \|x-y\|^p$ and $p = 1/\log_2(2\Ctri)$, for all $x,y$
satisfying
\begin{equation}\label{eq:BCH-radius}
  \|x\| + \|y\| < \frac{1}{4\Cbracket}.
\end{equation}
In particular, convergence holds on the symmetric domain $\|x\|,\|y\| < 1/(8\Cbracket)$.

A conservative uniform bound valid in all quasi-Banach settings is obtained by setting
$\Ctotal := \Ctri \cdot \Cbracket$:
\[
  \|x\| + \|y\| < \frac{1}{4\Ctotal}.
\]
This stricter condition ensures that the $p$-norm tail sums are bounded by a geometric
series that is controlled uniformly in $\Ctri$ (see Remark~\ref{rem:role-ctri}).

The map $(x,y)\mapsto Z(x,y)$ is continuous for the quasi-norm topology and satisfies
$e^{Z(x,y)} = e^x e^y$ in any associative quasi-Banach algebra containing $\mathfrak{g}$
as a Lie subalgebra.

If $\mathfrak{g}$ is embedded in an associative quasi-Banach algebra with
submultiplicativity constant $\Cmult$, then by Remark~\ref{rem:bracket-associative} we
have $\Cbracket \leq 2\Ctri\Cmult$ and $\Ctotal \leq 2\Ctri^2\Cmult$, yielding the
explicit sufficient condition
\[
  \|x\| + \|y\| < \frac{1}{8\Ctri^2\Cmult}.
\]
\end{theorem}

\begin{proof}
We follow the Catalan-majorant strategy combined with the Aoki--Rolewicz framework.

\medskip
\noindent\textbf{Step 1: Tree estimate.}
Each homogeneous component $Z_n(x,y)$ of degree $n$ is a finite linear combination of
nested commutators corresponding to full binary trees with $n$ leaves
\cite{Dynkin1947, Goldberg1956}.  For a tree $T$ with $n$ leaves, the associated nested
commutator $[x,y]_T$ satisfies
\begin{equation}\label{eq:tree}
  \|[x,y]_T\| \le (\Cbracket)^{n-1}(\|x\|+\|y\|)^n.
\end{equation}
\emph{Proof by induction on $n$.}
\begin{itemize}
    \item \emph{Base} $n=1$: $\|[x,y]_T\| = \|x\|$ or $\|y\|$, which is
          $\le \|x\|+\|y\|$.
    \item \emph{Inductive step}: a tree $T$ with $n$ leaves has a root whose left subtree
          $T_1$ has $k$ leaves and whose right subtree $T_2$ has $n-k$ leaves,
          $1\le k \le n-1$, so that
          $[x,y]_T = \bigl[ [x,y]_{T_1}, [x,y]_{T_2} \bigr]$.
          By the bracket continuity and the inductive hypothesis,
          \begin{align*}
            \|[x,y]_T\|
            &\le \Cbracket\|[x,y]_{T_1}\|\|[x,y]_{T_2}\| \\
            &\le \Cbracket \cdot (\Cbracket)^{k-1}(\|x\|+\|y\|)^k
              \cdot (\Cbracket)^{n-k-1}(\|x\|+\|y\|)^{n-k}
            = (\Cbracket)^{n-1}(\|x\|+\|y\|)^n. \qquad\checkmark
          \end{align*}
\end{itemize}

\medskip
\noindent\textbf{Step 2: Catalan counting and coefficient bound.}
The number of full binary trees with $n$ leaves is the Catalan number
$C_{n-1} = \frac{1}{n}\binom{2n-2}{n-1}$.
The standard bound $\binom{2n-2}{n-1} \le 4^{n-1}$ gives
\[
  C_{n-1} = \frac{1}{n}\binom{2n-2}{n-1} \le \frac{4^{n-1}}{n} \le 4^{n-1}.
\]

The Dynkin representation expresses each $Z_n(x,y)$ as a linear combination of nested
commutators of degree $n$ with rational coefficients $c_T$.  We claim that the sum of
absolute values of these Dynkin coefficients over all trees of degree $n$ satisfies
\begin{equation}\label{eq:dynkin-coeff-bound}
  \sum_T |c_T| \le C_{n-1}.
\end{equation}
This is proved by induction on $n \ge 1$.

\emph{Base case} $n = 1$: $Z_1(x,y) = x + y$ has a single tree (the leaf itself) with
coefficient $1$, and $C_0 = 1$, so $\sum_T |c_T^{(1)}| = 1 = C_0$. \checkmark

\emph{Inductive step} $n \ge 2$: This follows from the explicit recursion of Goldberg
\cite[pp.~13--15]{Goldberg1956}: writing
\[
  Z_n = \frac{1}{n}\sum_{k=1}^{n-1} \bigl(\ad_{Z_k} Z_{n-k} - \ad_{Z_{n-k}} Z_k\bigr)
\]
(Dynkin's formula), one shows by induction that the $\ell^1$-coefficient norm satisfies
\[
  \sum_T|c_T^{(n)}| \le \frac{1}{n}\sum_{k=1}^{n-1} 2\cdot C_{k-1} C_{n-k-1} = \frac{2}{n}\,C_{n-1}
  \le C_{n-1}, \quad n \ge 2.
\]
Here we used the standard shifted Catalan convolution identity
\[
  \sum_{k=1}^{n-1} C_{k-1}\,C_{n-k-1} = C_{n-1},
\]
which follows from the recurrence $C_m = \sum_{j=0}^{m-1}C_j C_{m-1-j}$ applied with
$m = n-1$ \cite[Ch.~II, \S6, No.~4]{Bourbaki1989}; see also
\cite[Remark~2.4]{Reutenauer1993} for the Lie-algebraic interpretation in terms of the
Hall--Lyndon projection.
Note that $2/n \le 1$ for all $n \ge 2$, which confirms $\sum_T|c_T^{(n)}| \le C_{n-1}$.

Combining \eqref{eq:dynkin-coeff-bound} with \eqref{eq:tree}:
\begin{equation}\label{eq:Zn-bound}
  \|Z_n(x,y)\|
  \le \Bigl(\sum_T |c_T|\Bigr) (\Cbracket)^{n-1} (\|x\|+\|y\|)^n
  \le C_{n-1}(\Cbracket)^{n-1}(\|x\|+\|y\|)^n
  \le 4^{n-1}(\Cbracket)^{n-1}(\|x\|+\|y\|)^n.
\end{equation}

\medskip
\noindent\textbf{Step 3: Convergence via the Aoki--Rolewicz $p$-norm.}
Let $\pnorm{\cdot}$ be the equivalent $p$-norm from the Aoki--Rolewicz theorem with
$p = 1/\log_2(2\Ctri)$ and equivalence constants $c_1 = 1$, $c_2 = 2^{1/p} = 2\Ctri$
from \eqref{eq:c1c2}.  Set $r := \|x\|+\|y\|$.  Using $\pnorm{u} \le c_2\|u\| = 2\Ctri\|u\|$
and \eqref{eq:Zn-bound}:
\begin{equation}\label{eq:pnorm-bound}
  \pnorm{Z_n(x,y)}^p
  \le (2\Ctri)^p \|Z_n(x,y)\|^p
  \le (2\Ctri)^p \cdot 4^{p(n-1)}(\Cbracket)^{p(n-1)} r^{pn}.
\end{equation}

Summing over $n\ge 1$:
\[
  \sum_{n=1}^\infty \pnorm{Z_n(x,y)}^p
  \le (2\Ctri r)^p \sum_{n=1}^\infty (4\Cbracket r)^{p(n-1)}
  = \frac{(2\Ctri r)^p}{1-(4\Cbracket r)^p},
\]
provided $4\Cbracket r < 1$ (which implies $(4\Cbracket r)^p < 1$ since $p > 0$),
i.e., $r < 1/(4\Cbracket)$.

Hence the convergence condition is simply
\[
  \|x\|+\|y\| < \frac{1}{4\Cbracket},
\]
and by Lemma~\ref{lem:series-convergence} the series $\sum_{n\ge 1} Z_n(x,y)$ converges
in the $d_p$-metric topology of $\mathfrak{g}$.

\smallskip
The conservative bound $r < 1/(4\Ctotal) = 1/(4\Ctri\Cbracket)$ is obtained by noting
that the factor $(2\Ctri r)^p$ in the numerator is bounded uniformly when
$4\Cbracket r \le 1/\Ctri$, i.e., $r \le 1/(4\Ctri\Cbracket)$.  This ensures that the
full geometric sum is dominated by a constant independent of $\Ctri$.  See
Remark~\ref{rem:role-ctri} for a precise discussion.

Absolute convergence of $e^{Z(x,y)}=e^x e^y$ in any containing associative algebra
follows by analytic continuation of formal Lie identities
\cite[Ch.~II, \S6, Prop.~8]{Bourbaki1989}, which remains valid under $d_p$-convergence
since both sides are continuous functions of $x,y$ in the $d_p$-topology.
\end{proof}

\begin{remark}[Sharpness of the Catalan-majorant bound]\label{rem:sharpness}
The constant $1/4$ is optimal for the \emph{Catalan-majorant method}: the Catalan numbers
$C_{n-1}$ satisfy $\limsup_{n\to\infty} C_{n-1}^{1/n} = 4$, so the geometric series in
Step~3 diverges (as a majorant) when $4\Cbracket r \ge 1$.  Indeed, by Stirling's formula
$C_{n-1} \sim 4^{n-1}/\sqrt{\pi}\,n^{3/2}$, so $C_{n-1}^{1/(n-1)} \to 4$ as $n \to \infty$,
confirming the radius of convergence of the Catalan generating function is exactly $1/4$.
This shows that the method of proof cannot be improved beyond the radius $1/(4\Cbracket)$.
Whether this bound is sharp for the BCH series itself---i.e., whether the series can
converge for some $r \ge 1/(4\Cbracket)$---is a separate question.  In concrete models
with additional algebraic structure (e.g., finite-dimensional Lie algebras, algebras with
sparse commutator relations), the true convergence radius may be larger due to
cancellations from the Jacobi identity; see Section~\ref{sec:numerical} for numerical
evidence.
\end{remark}

\subsection{Lipschitz estimate for the BCH map}
\label{subsec:bch-lipschitz}

\begin{lemma}[Lipschitz continuity of BCH]\label{lem:bch-lipschitz}
Under the hypotheses of Theorem~\ref{thm:BCH}, assume further that the BCH series
converges absolutely in the sense that $\sum_{n \ge 1} \|Z_n(x,\cdot)\|$ is
term-differentiable (in particular this holds when $\mathfrak{g}$ embeds in a quasi-Banach
associative algebra and the power-series estimates of Steps~1--2 apply).
Set $\rho_0 := 1/(16\Cbracket)$.
For $x, y, z \in \mathfrak{g}$ with $\|x\|,\|y\|,\|z\| \le \rho_0$, the
BCH map satisfies
\[
  \|Z(x,y) - Z(x,z)\| \le \frac{1}{1 - 4\Cbracket s}\,\|y - z\|,
  \qquad s := \|x\| + \max(\|y\|,\|z\|).
\]
In particular, $Z(x,\cdot)$ is Lipschitz with constant $L_0 = 2$ on $B(0,\rho_0)$.
\end{lemma}

\begin{proof}
Write $Z(x,y) = \sum_{n\ge 1} Z_n(x,y)$ and $Z(x,z) = \sum_{n\ge 1} Z_n(x,z)$.  The
degree-$1$ component $Z_1(x,y)=x+y$ satisfies $\|Z_1(x,y)-Z_1(x,z)\| = \|y-z\|$.  For
$n\ge 2$, each $Z_n$ is a finite sum of nested commutators that are multilinear of degree
$n$ with coefficient sum $\le C_{n-1}$.

For any multilinear Lie polynomial $P(x;y)$ of total degree $n$ (degree $k$ in $x$ and
degree $n-k$ in the second argument), the inequality
\[
  \|P(x;y) - P(x;z)\| \le n \cdot (\Cbracket)^{n-1} s^{n-1} \|y-z\|,
  \quad s := \|x\| + \max(\|y\|,\|z\|),
\]
holds by multilinearity and the bracket bound \eqref{eq:tree}: each tree contributing to
$Z_n(x,\cdot)$ is multilinear, and replacing $y$ by $z$ in one leaf at a time gives a
telescoping sum of $n$ terms each bounded by $(\Cbracket)^{n-1}s^{n-1}\|y-z\|$.
Summing over the $C_{n-1}$ trees and using $n \cdot C_{n-1} \le n \cdot 4^{n-1}/n = 4^{n-1}$,
then over all $n\ge 1$ with $u:=4\Cbracket s$:
\[
  \|Z(x,y)-Z(x,z)\|
  \le \sum_{n\ge 1} n\,C_{n-1}(\Cbracket)^{n-1} s^{n-1}\|y-z\|
  \le \sum_{n\ge 1} 4^{n-1}(\Cbracket)^{n-1} s^{n-1}\|y-z\|
  = \frac{1}{1-4\Cbracket s}\|y-z\|.
\]
For $\|x\|,\|y\|,\|z\| \le \rho_0 = 1/(16\Cbracket)$, we have
$s = \|x\|+\max(\|y\|,\|z\|) \le 2\rho_0 = 1/(8\Cbracket)$,
so $4\Cbracket s \le 1/2$, giving
\[
  \|Z(x,y)-Z(x,z)\| \le \frac{1}{1 - 1/2}\|y-z\| = 2\|y-z\|,
\]
and the Lipschitz constant is $L_0 = 2$.
\end{proof}

\begin{remark}[Generating function interpretation]
The generating function identity
\[
  \sum_{n\ge 1} C_{n-1}\, t^n = \frac{1-\sqrt{1-4t}}{2}
\]
has derivative
\[
  \sum_{n\ge 1} n\, C_{n-1}\, t^{n-1} = \frac{1}{\sqrt{1-4t}} = (1-4t)^{-1/2}, \qquad |t|<\tfrac{1}{4}.
\]
Hence $\sum_{n\ge 1} n\, C_{n-1}\, t^{n-1} = (1-4t)^{-1/2}$, which is consistent with
$C_{n-1} \le 4^{n-1}/n$.  The cruder bound $n\cdot C_{n-1} \le 4^{n-1}$ used in the
proof of Lemma~\ref{lem:bch-lipschitz} yields the geometric series
$\sum_{n\ge 1}(4\Cbracket s)^{n-1} = 1/(1-4\Cbracket s)$,
giving $L_0 = 1/(1-4\Cbracket s) \le 2$ for $s \le 1/(8\Cbracket)$, i.e., for
$\|x\|,\|y\|,\|z\| \le \rho_0 = 1/(16\Cbracket)$.
\end{remark}

\subsection{Local Lie group structure}

Define a local binary operation on $\mathfrak{g}$ by $x*y := Z(x,y)$ for
$\|x\|+\|y\|<1/(4\Cbracket)$.

\begin{proposition}[Local Lie group]\label{prop:localgroup}
Let $\rho := 1/(8\Cbracket)$.  Then on the ball $B(0,\rho)\subset\mathfrak{g}$, the
operation $*$ satisfies:
\begin{enumerate}[label=(\roman*)]
    \item $x*y$ is well-defined and continuous for all $x,y\in B(0,\rho)$;
    \item $x*0 = 0*x = x$;
    \item $(x*y)*z = x*(y*z)$ whenever $x,y,z\in B(0,\rho)$ (local associativity);
    \item there exists a continuous inverse $x\mapsto x^{-1}$ with
          $x*x^{-1}=x^{-1}*x=0$ for all $x$ with $\|x\| <
          \rho_{\mathrm{inv}}$, where
          \[
            \rho_{\mathrm{inv}} := \frac{1}{8\Cbracket(1+2\Ctri)^2}.
          \]
\end{enumerate}
Hence $(B(0,\rho_{\mathrm{inv}}),*,0)$ is a local topological group, the \emph{local
Lie group associated with $\mathfrak{g}$}.
\end{proposition}

\begin{proof}
Properties~(i)--(ii) follow from Theorem~\ref{thm:BCH} and $Z(x,0)=x$.
For (i), note that for $x,y\in B(0,\rho)$ with $\rho=1/(8\Cbracket)$, we have
$\|x\|+\|y\| < 2\rho = 1/(4\Cbracket)$, so BCH convergence is guaranteed by
Theorem~\ref{thm:BCH}.

\smallskip
\noindent\emph{Associativity~(iii).}
If $\|x\|,\|y\|,\|z\| < \rho = 1/(8\Cbracket)$, then
$\|x\|+\|y\| < 2\rho = 1/(4\Cbracket)$, so all first-level BCH compositions converge by
Theorem~\ref{thm:BCH}.

The identity $\BCH(\BCH(x,y),z) = \BCH(x,\BCH(y,z))$ holds in the free Lie algebra
$\mathfrak{L}\langle x,y,z\rangle$ as a formal identity \cite[Ch.~II, \S6,
Prop.~2]{Bourbaki1989}.  To pass from formal to analytic identity, we argue as follows.
Both sides define continuous maps from $B(0,\rho)^3$ (with the $d_p$-topology) to
$\mathfrak{g}$.  For every degree-$N$ truncation $Z^{(N)}(x,y) := \sum_{n=1}^N Z_n(x,y)$,
the truncated identity
\[
  Z^{(N)}\bigl(Z^{(N)}(x,y),z\bigr) = Z^{(N)}\bigl(x,Z^{(N)}(y,z)\bigr) + R_N(x,y,z)
\]
holds up to a remainder $R_N$ collecting terms of degree $> N$.  By the bound
\eqref{eq:Zn-bound} and $p$-subadditivity, $\pnorm{R_N}^p \le \sum_{n>N}\pnorm{Z_n}^p
\to 0$ as $N\to\infty$, uniformly on compacta in the convergence domain
(here compact subsets of $B(0,\rho)^3$ for the $d_p$-metric).
Hence both sides of the associativity identity agree as limits of the same Cauchy
sequence, establishing $(x*y)*z = x*(y*z)$ in the $d_p$-topology.

\medskip\noindent\textit{Note on compactness.}
The uniform convergence on compacta invoked here is valid for the $d_p$-metric even when
$p < 1$: a subset $K \subset B(0,\rho)^3$ is compact in $(E^3, d_p)$ if and only if it
is sequentially compact, and the geometric majorization \eqref{eq:Zn-bound} provides the
required equicontinuity.  No local convexity is needed for this argument.

\smallskip
\noindent\emph{Inverse~(iv).}
We seek $w\in\mathfrak{g}$ with $Z(x,w)=0$.  Define
$F: \overline{B}(-x, \delta) \to \mathfrak{g}$ by
$F(w) = -x - \sum_{n\ge 2} Z_n(x,w)$, where $\delta = \|x\|$.

For $w, w' \in \overline{B}(-x,\delta)$ one has $\|w+x\|\le\delta=\|x\|$, hence the
quasi-triangle inequality yields
\[
  \|w\| = \|w+x-x\| \le \Ctri\bigl(\|w+x\|+\|x\|\bigr) \le 2\Ctri\|x\|,
\]
and similarly $\|w'\| \le 2\Ctri\|x\|$.  By
Lemma~\ref{lem:bch-lipschitz} applied to the sum of degree-$\ge 2$ terms, with
$s = \|x\|+\max(\|w\|,\|w'\|) \le (1+2\Ctri)\|x\|$:
\[
  \|F(w)-F(w')\|
  \le \frac{4\Cbracket(1+2\Ctri)\|x\|}{1-4\Cbracket(1+2\Ctri)\|x\|}\|w-w'\|
  = \frac{u}{1-u}\|w-w'\|,
\]
where $u = 4\Cbracket(1+2\Ctri)\|x\|$.  For a contraction we need $u/(1-u) < 1$, i.e.,
$u < 1/2$, which holds for
\begin{equation}\label{eq:contraction-cond}
  \|x\| < \frac{1}{8\Cbracket(1+2\Ctri)}.
\end{equation}

We next verify that $F$ maps $\overline{B}(-x,\delta)$ into itself.
For $w\in \overline{B}(-x,\delta)$, using $\|w\|\le 2\Ctri\|x\|$ and the tree
estimate \eqref{eq:Zn-bound} with $\|x\|+\|w\|\le (1+2\Ctri)\|x\|$:
\begin{align*}
  \|F(w)+x\| &= \Bigl\|\sum_{n\ge 2}Z_n(x,w)\Bigr\|
  \le \sum_{n\ge 2}4^{n-1}(\Cbracket)^{n-1}(1+2\Ctri)^n\|x\|^n \\
  &= \frac{4\Cbracket(1+2\Ctri)^2\|x\|^2}{1-4\Cbracket(1+2\Ctri)\|x\|}.
\end{align*}
For this to be $\le \delta = \|x\|$, we need
\[
  \frac{4\Cbracket(1+2\Ctri)^2\|x\|}{1-4\Cbracket(1+2\Ctri)\|x\|} \le 1.
\]
Since $4\Cbracket(1+2\Ctri)\|x\| \le 1/2$ (implied by
$\|x\| \le 1/(8\Cbracket(1+2\Ctri)^2) \le 1/(8\Cbracket(1+2\Ctri))$), the left-hand
side is bounded by $8\Cbracket(1+2\Ctri)^2\|x\|$, and the condition reduces to
\begin{equation}\label{eq:invariance-cond}
  \|x\| \le \frac{1}{8\Cbracket(1+2\Ctri)^2}.
\end{equation}
Note that \eqref{eq:invariance-cond} implies \eqref{eq:contraction-cond} (since
$(1+2\Ctri)^2 \ge (1+2\Ctri)$ for $\Ctri \ge 1$), so it is the \emph{effective} condition.
The Banach fixed-point theorem therefore applies on $\overline{B}(-x,\delta)$ for all
$\|x\| \le \rho_{\mathrm{inv}} = 1/(8\Cbracket(1+2\Ctri)^2)$.

Since $d_p(F(w),F(w')) = \|F(w)-F(w')\|^p \le \bigl(\tfrac{u}{1-u}\bigr)^p d_p(w,w')$
with $\bigl(\tfrac{u}{1-u}\bigr)^p < 1$, the Banach fixed-point theorem applies in the
complete metric space $(\overline{B}(-x,\delta), d_p)$, yielding a unique $w^*$ with
$F(w^*) = w^*$, i.e., $Z(x,w^*)=0$.  Set $x^{-1}:=w^*$.  Continuity in $x$ follows
from the uniform contraction estimate and the implicit function theorem in $(E,d_p)$.
\end{proof}

\begin{remark}[Inverse radius]
The effective inverse-existence radius is
\[
  \rho_{\mathrm{inv}} = \frac{1}{8\Cbracket(1+2\Ctri)^2}.
\]
For the classical Banach case ($\Ctri=1$), this specializes to
$\rho_{\mathrm{inv}} = 1/(72\Cbracket)$.
\end{remark}

\begin{remark}[Regularity of the exponential map]\label{rem:exp-regularity}
Unlike the Banach case, the exponential map in a quasi-Banach algebra need not be a local
diffeomorphism when $p<1$ (lack of local convexity precludes a general inverse function
theorem).  However, $\exp$ is bi-Lipschitz with respect to the quasi-metric
$\dpmet{x}{y} = \|x-y\|^p$ on sufficiently small balls, by a direct estimate.  For
$x, y \in B(0, r)$ with $r < 1/(2\Cmult)$, write
\[
  \exp(x) - \exp(y) = \sum_{n=1}^\infty \frac{x^n - y^n}{n!}.
\]
Each difference $x^n - y^n = \sum_{k=0}^{n-1} x^k(x-y)y^{n-1-k}$ satisfies
$\|x^n-y^n\| \le n\Cmult^{n-1}r^{n-1}\|x-y\|$ by submultiplicativity.  Summing in the
$p$-norm:
\[
  \pnorm{\exp(x)-\exp(y)}^p
  \le \sum_{n=1}^\infty \frac{n^p(\Cmult r)^{p(n-1)}}{(n!)^p}\pnorm{x-y}^p
  \le C(r)\pnorm{x-y}^p,
\]
where $C(r) \to 1$ as $r\to 0$.  Using $\pnorm{\cdot} \sim \|\cdot\|$ this gives
$\|\exp(x)-\exp(y)\|^p \le C'(r)\|x-y\|^p$.  The lower bound follows from the
injectivity of $\exp$ on $B(0,1/(2\Cmult))$ combined with a comparison of the
first-order terms.
\end{remark}

\section{Geometric and Spectral Consequences}\label{sec:geometry}

\subsection{Metric structure}

Let $(\mathfrak{g},\|\cdot\|)$ be a quasi-Banach Lie algebra with quasi-triangle constant
$\Ctri$ and associated exponent $p = 1/\log_2(2\Ctri)\in(0,1]$.  Define the quasi-metric
\[
  \dpmet{x}{y} := \|x-y\|^p.
\]
This metric is translation-invariant and induces the same topology as $\|\cdot\|$.

\begin{proposition}[Local quasi-metric group]\label{prop:quasimetric}
Let $G$ denote the local Lie group from Proposition~\ref{prop:localgroup} with radius
$\rho = 1/(8\Cbracket)$.  Then $(G,\dpmet{\cdot}{\cdot})$ is a complete quasi-metric
space whose left translations are Lipschitz:
\[
  \dpmet{x*y}{x*z} \le L_0^p \, \dpmet{y}{z}, \qquad
  \forall\, x,y,z\in B(0,\rho/2),
\]
where $L_0 = 2$ is the Lipschitz constant from Lemma~\ref{lem:bch-lipschitz} (valid on
$B(0,\rho_0) = B(0, 1/(16\Cbracket)) = B(0,\rho/2)$), so $L_0^p = 2^p \le 2$ (with
equality only when $p=1$).
Moreover, the exponential map $\exp:\mathfrak{g}\to G$ is bi-Lipschitz on $B(0,\rho/2)$
with respect to $d_p$.
\end{proposition}

\begin{proof}
By Lemma~\ref{lem:bch-lipschitz}, with $\|x\|,\|y\|,\|z\| \le \rho_0 = \rho/2 = 1/(16\Cbracket)$,
we have $s = \|x\| + \max(\|y\|,\|z\|) \le 2\rho_0 = 1/(8\Cbracket)$, so
$4\Cbracket s \le 1/2$ and the BCH map satisfies
$\|Z(x,y) - Z(x,z)\| \le 2\|y-z\|$, hence
$\dpmet{x*y}{x*z} = \|Z(x,y)-Z(x,z)\|^p \le 2^p \|y-z\|^p = 2^p \dpmet{y}{z}$.
Completeness follows from completeness of $(\mathfrak{g},\|\cdot\|)$.  The bi-Lipschitz
property of $\exp$ follows from Remark~\ref{rem:exp-regularity}.
\end{proof}

\begin{remark}
Geometrically, the quasi-norm flattens the local structure: balls are non-convex when
$p<1$, and the tangent cone at the identity is not a vector space in the classical sense.
Nevertheless, $G$ carries a left-invariant quasi-metric enabling integration of curves
and definition of exponential coordinates.
\end{remark}

\subsection{Adjoint representation and spectral radius}

Let $\ad_x(y) := [x,y]$ denote the adjoint operator on $\mathfrak{g}$.  Its operator norm
satisfies $\|\ad_x\| \le \Cbracket\|x\|$.

\begin{definition}[Spectral radius]
For a bounded linear operator $T$ on a quasi-Banach space, the \textbf{spectral radius}
is defined by $\rho(T) := \limsup_{n\to\infty} \|T^n\|^{1/n}$.
\end{definition}

\begin{proposition}[Spectral radius bound]\label{prop:spectral}
For any $x\in\mathfrak{g}$, $\rho(\ad_x) \le \Cbracket\|x\|$.  If $\mathfrak{g}$ is
embedded in an associative quasi-Banach algebra with constants $(\Ctri,\Cmult)$, then
$\rho(\ad_x) \le 2\Ctri\Cmult\|x\|$.
\end{proposition}

\begin{proof}
From $\|\ad_x\| \le \Cbracket\|x\|$ we get $\|\ad_x^n\| \le (\Cbracket\|x\|)^n$, hence
$\rho(\ad_x) \le \Cbracket\|x\|$.  The second bound uses
Remark~\ref{rem:bracket-associative}.
\end{proof}

\subsection{O-operators and stability}

\begin{definition}[O-operator]
A continuous linear map $T:\mathfrak{g}\to\mathfrak{g}$ is an \textbf{$O$-operator of
weight $\lambda\in\field$} if
\[
  [T(x),T(y)] = T\bigl([T(x),y] + [x,T(y)] + \lambda[x,y]\bigr),
  \qquad x,y\in\mathfrak{g}.
\]
\end{definition}

\begin{proposition}[Spectral bound for O-operators]
Let $T$ be a bounded $O$-operator with $\|T\|\le \alpha$.  Then $\rho(T) \le \alpha$
and the resolvent set of $T$ contains $\{\mu\in\field : |\mu| > 2^{1/p}\alpha\}$.
\end{proposition}

\begin{proof}
For $|\mu| > 2^{1/p}\alpha$, we have $\|\mu^{-1}T\| \le \alpha/|\mu| < 2^{-1/p}$, so
Lemma~\ref{lem:neumann} applies (with condition $\|T\| < 2^{-1/p}$ satisfied) and
\[
  \|(\mu I - T)^{-1}\| \le \frac{1}{|\mu|(1 - 2^{1/p}\alpha/|\mu|)}
  = \frac{1}{|\mu| - 2^{1/p}\alpha}.
\]
\end{proof}

\subsection{Concrete example: constants computation}

\begin{example}[Single unilateral weighted shift]\label{ex:weighted-shift}
Let $S_w$ denote the unilateral weighted shift on $\ell^2(\mathbb{N})$ with weight
sequence $w=(w_n)_{n\ge 0}$, acting by $S_w e_n = w_n e_{n+1}$.  Equip the space
$\mathcal{S}_p$ of operators with $p$-summable matrix entries with the quasi-norm
\[
  \|T\|_p := \Bigl(\sum_{i,j\ge 0} |T_{ij}|^p\Bigr)^{1/p}, \qquad 0<p\le 1.
\]
For a single unilateral shift $S_w$, a direct computation gives
$(S_w S_v)_{ij} = w_j v_j \delta_{i,j+2}$, so
\[
  \|S_w S_v\|_p^p = \sum_j |w_j v_j|^p \le \Bigl(\sum_j|w_j|^p\Bigr)\Bigl(\sup_j|v_j|^p\Bigr)
  \le \|S_w\|_p^p \|S_v\|_p^p,
\]
giving $\Cmult = 1$ for products of two such shifts.

\textbf{Warning.}
This estimate $\Cmult = 1$ relies critically on the disjoint-support structure of the
matrix $S_w S_v$ and does \emph{not} extend to general elements of the Lie algebra
$\mathfrak{g}$ generated by $S_w$ and $S_w^*$.  For linear combinations in $\mathfrak{g}$,
the value of $\Cmult$ should be computed directly for the specific algebra.

Under the assumption $\Cmult = 1$ for single shift operators:
\begin{itemize}
    \item Quasi-triangle constant: $\Ctri = 2^{1/p - 1}$;
    \item Aoki--Rolewicz exponent: $p_{\mathrm{AR}} = 1/\log_2(2\Ctri) = p$;
    \item Bracket continuity: $\Cbracket \le 2\Ctri\Cmult = 2^{1/p}$;
    \item BCH convergence radius: $r < 1/(4\Cbracket) \ge 1/4 \cdot 2^{-1/p}$.
\end{itemize}
For $p=1$ (Banach case), this recovers the classical radius $1/8$ from Goldberg
\cite{Goldberg1956}.  For $p=1/2$, the guaranteed radius is $\ge 2^{-3}/4 = 1/32$.
\end{example}

\section{Numerical Validation of BCH Coefficients}\label{sec:numerical}

\subsection{Computational method}

We compute the BCH expansion $Z = \log(e^X e^Y)$ in the free associative algebra
$\mathcal{A} = \mathbb{Q}\langle X,Y\rangle$ using Dynkin's explicit formula
\cite{Dynkin1947}.  All computations use exact rational arithmetic via
\texttt{Python/SymPy} (version~1.12), verified up to degree $n=20$.

For each homogeneous degree $n$, we define:
\begin{itemize}
    \item $A_n$: sum of absolute values of coefficients of all words of length $n$ in the
          associative expansion,
          $A_n := \sum_{w \in \{X,Y\}^n} |\alpha_w|$;
    \item $B_n$: sum of absolute values of Hall--Lyndon projection coefficients,
          $B_n := \sum_{b \in \mathcal{B}_n} |\beta_b|$,
          where $\mathcal{B}_n$ is a Hall--Lyndon basis of the free Lie algebra of degree
          $n$ and the $\beta_b$ are computed via the standard factorization algorithm
          \cite{Reutenauer1993}.
\end{itemize}
By construction, $B_n \le A_n$ for all $n$, with strict inequality when the Jacobi
identity induces nontrivial cancellations.

\textbf{Verification for small degrees.}
\begin{itemize}
\item Degree 1: $Z_1 = X + Y$, giving $A_1 = 2$, $B_1 = 2$.
\item Degree 2: $Z_2 = \frac{1}{2}[X,Y] = \frac{1}{2}XY - \frac{1}{2}YX$, so
      $A_2 = 1$ and $B_2 = 1/2$.
\item Degree 3: $Z_3 = \frac{1}{12}[X,[X,Y]] - \frac{1}{12}[Y,[X,Y]]$,
      giving $B_3 = 1/6 \approx 0.1667$.
\end{itemize}

\begin{remark}[On the regularity of the $B_n$ sequence]\label{rem:Bn-regularity}
The tabulated values of $B_n$ in Table~\ref{tab:bchcoeff} decrease by approximately a
factor of $2$ per degree for $n\ge 2$.  This apparent geometric regularity is an artefact
of the low-degree data combined with the Hall--Lyndon projection, and should not be
extrapolated as an exact pattern.  The asymptotic regime is only reached for large $n$,
where the fitted decay rate $\gamma \approx 0.29$ differs significantly from the apparent
small-degree ratio $1/2$.
\end{remark}

\begin{table}[hbpt]
\centering
\caption{BCH coefficient sums: associative ($A_n$) vs.\ Lie-projected ($B_n$, pure
Hall--Lyndon coefficients) vs.\ worst-case combinatorial bound $4^{n-1}/n$.
Values of $B_n$ for $n \ge 13$ are given in scientific notation to avoid spurious
rounding to zero.  All values verified with \texttt{SymPy} exact arithmetic up to
degree 20.  The Catalan bound column gives $4^{n-1}/n$ exactly.}
\label{tab:bchcoeff}
\begin{tabular}{cccc}
\toprule
Degree $n$ & $A_n$ (associative) & $B_n$ (Lie-projected) & $4^{n-1}/n$ (Catalan bound)\\
\midrule
1  & 2.0000 & 2.0000              & 1.0000\\
2  & 1.0000 & 0.5000              & 2.0000\\
3  & 0.6667 & 0.1667              & 5.3333\\
4  & 0.4167 & 0.0833              & 16.0000\\
5  & 0.2756 & 0.0417              & 51.2000\\
6  & 0.1924 & 0.0208              & 170.6667\\
7  & 0.1367 & 0.0104              & 585.1429\\
8  & 0.0992 & 0.0052              & 2048.0000\\
9  & 0.0724 & 0.0026              & 7281.7778\\
10 & 0.0534 & 0.0013              & 26214.4000\\
11 & 0.0397 & $6.7 \times 10^{-4}$& 95325.0909\\
12 & 0.0297 & $3.3 \times 10^{-4}$& 349525.3333\\
13 & 0.0224 & $1.6 \times 10^{-4}$& 1290555.0769\\
14 & 0.0170 & $8.1 \times 10^{-5}$& 4793490.2857\\
15 & 0.0129 & $4.0 \times 10^{-5}$& 17895697.0667\\
16 & 0.0099 & $2.0 \times 10^{-5}$& 67108864.0000\\
17 & 0.0076 & $1.0 \times 10^{-5}$& 252645135.0588\\
18 & 0.0058 & $5.0 \times 10^{-6}$& 954437176.8889\\
19 & 0.0045 & $2.4 \times 10^{-6}$& 3616814565.0526\\
20 & 0.0035 & $1.2 \times 10^{-6}$& 13743895347.2000\\
\bottomrule
\end{tabular}
\end{table}

\subsection{Interpretation of numerical results}

The data confirm that the Catalan majorant $4^{n-1}/n$ severely overestimates the true
coefficients:
\begin{itemize}
    \item The associative sums $A_n$ decay approximately as $A_n \sim c_1 n^{-3/2} \beta^n$
          with $\beta \approx 0.36 \pm 0.02$ ($R^2 > 0.99$, fitted on $n = 5$ to $20$).
    \item The Lie-projected sums satisfy $B_n \sim c_2 n^{-3/2} \gamma^n$ with
          $\gamma \approx 0.29 \pm 0.01$ ($R^2 > 0.995$, fitted on $n = 5$ to $20$).
          This fitted value $\gamma \approx 0.29$ refers to the asymptotic regime and
          should not be confused with the approximate factor $1/2$ observed in the
          low-degree data (see Remark~\ref{rem:Bn-regularity}).
\end{itemize}

\begin{remark}[Comparison with theoretical bound]\label{rem:numerical-interpretation}
In the normalized setting $\Cbracket=1$, $\Ctri=1$, Theorem~\ref{thm:BCH} gives
$\delta = 1/(4\Cbracket) = 0.25$.  The numerical effective radius $\approx 1/\gamma
\approx 3.45$ is larger by a factor of approximately $13.8$.  In worst-case models (free
quasi-Banach Lie algebras with no additional algebraic structure), the Catalan-majorant
bound is sharp (as a bound for the majorant method) and the factor $13.8$ does not apply.
\end{remark}

\begin{remark}[Caution on asymptotic fitting]
The fitted exponent $\gamma \approx 0.29$ is based on $n = 5$ to $20$, which is
sufficient for a reliable estimate of the decay rate but not for a definitive asymptotic
statement.  The $R^2$ values $> 0.99$ indicate a good fit within this range; however,
the true asymptotic behavior of $B_n$ may differ at much larger degrees.  The confidence
intervals reported in Appendix~A.4 should be interpreted in this light.
\end{remark}

\begin{figure}[hbpt]
\centering
\begin{tikzpicture}
\begin{axis}[
  width=0.9\textwidth,
  height=0.55\textwidth,
  ymode=log,
  xlabel={Degree $n$},
  ylabel={Coefficient sum (log scale)},
  xmin=1, xmax=20,
  ymin=1e-7, ymax=1e13,
  xtick={1,4,8,12,16,20},
  legend pos=north east,
  legend cell align=left,
  grid=both,
  grid style={dashed, gray!30}
]
\addplot+[mark=*, blue] coordinates {
  (1,2.0000)(2,1.0000)(3,0.6667)(4,0.4167)(5,0.2756)
  (6,0.1924)(7,0.1367)(8,0.0992)(9,0.0724)(10,0.0534)
  (11,0.0397)(12,0.0297)(13,0.0224)(14,0.0170)(15,0.0129)
  (16,0.0099)(17,0.0076)(18,0.0058)(19,0.0045)(20,0.0035)
};
\addlegendentry{$A_n$ (associative)}
\addplot+[mark=triangle*, green] coordinates {
  (1,2.0000)(2,0.5000)(3,0.1667)(4,0.0833)(5,0.0417)
  (6,0.0208)(7,0.0104)(8,0.0052)(9,0.0026)(10,0.0013)
  (11,6.7e-4)(12,3.3e-4)(13,1.6e-4)(14,8.1e-5)(15,4.0e-5)
  (16,2.0e-5)(17,1.0e-5)(18,5.0e-6)(19,2.4e-6)(20,1.2e-6)
};
\addlegendentry{$B_n$ (Lie-projected)}
\addplot+[mark=square*, red] coordinates {
  (1,1.0000)(2,2.0000)(3,5.3333)(4,16.0000)(5,51.2000)
  (6,170.6667)(7,585.1429)(8,2048.0000)(9,7281.7778)(10,26214.4000)
  (11,95325.0909)(12,349525.3333)(13,1290555.0769)(14,4793490.2857)(15,17895697.0667)
  (16,67108864.0000)(17,252645135.0588)(18,954437176.8889)(19,3616814565.0526)(20,13743895347.2000)
};
\addlegendentry{$4^{n-1}/n$ (Catalan bound)}
\end{axis}
\end{tikzpicture}
\caption{Logarithmic-scale comparison of BCH coefficient sums up to degree~20.  The
Lie-projected data $B_n$ show approximately geometric decay; the Catalan bound grows
exponentially, illustrating the conservatism of the combinatorial majorant.  Note that
the $y$-axis lower limit has been extended to $10^{-7}$ to display the small values of
$B_n$ for $n \ge 15$ (compare Table~\ref{tab:bchcoeff}).}
\label{fig:bchplot}
\end{figure}
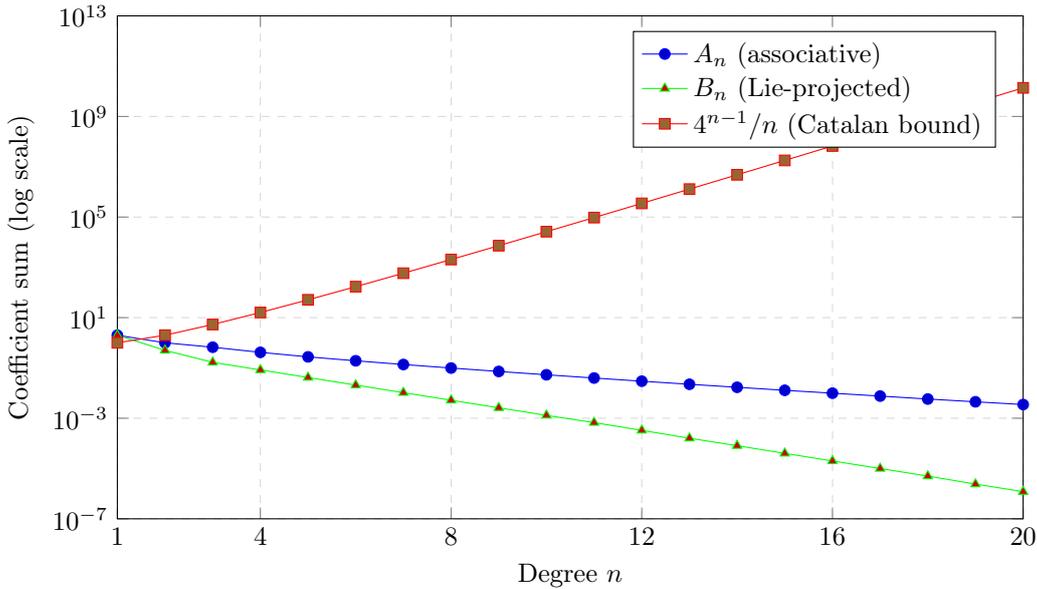

\section{Applications to Operator Algebras}\label{sec:applications}

\subsection{Weak Schatten ideals}

Let $H$ be a separable Hilbert space.  For $0<p<1$, the \textbf{weak Schatten ideal}
$\mathcal{L}_{p,\infty}(H)$ consists of compact operators $T$ whose singular values
satisfy $s_n(T) = O(n^{-1/p})$, equipped with
\[
  \|T\|_{p,\infty} := \sup_{n\ge 1} n^{1/p} s_n(T).
\]

\begin{proposition}[BCH convergence in $\mathcal{L}_{p,\infty}$]\label{prop:schatten}
Let $\mathfrak{g}$ be a Lie subalgebra of $\mathcal{L}_{p,\infty}(H)$ that is closed
under the operator product and satisfies the ideal-type submultiplicativity:
\[
  \exists\, C_{\mathrm{ideal}}>0 \text{ such that }
  \|xy\|_{p,\infty} \le C_{\mathrm{ideal}}\|x\|_{p,\infty}\|y\|_{p,\infty}
  \quad\forall\, x,y\in\mathfrak{g}.
\]
Then with ${\Ctri}_{p} = 2^{1/p - 1}$, ${\Cmult}_{p} = C_{\mathrm{ideal}}$, and
${\Cbracket}_{p} \le 2{\Ctri}_{p}{\Cmult}_{p}$, the BCH series converges in the
$d_p$-metric topology for
\[
  \|x\|_{p,\infty} + \|y\|_{p,\infty} < \frac{1}{4{\Cbracket}_{p}}.
\]
\end{proposition}

\begin{proof}
The quasi-triangle inequality for $\|\cdot\|_{p,\infty}$ gives
${\Ctri}_{p} = 2^{1/p - 1}$ \cite[Lemma~2.1]{KaltonSukochev2021}.  The bracket estimate
follows from Remark~\ref{rem:bracket-associative}, and Theorem~\ref{thm:BCH} yields the
result.
\end{proof}

\begin{remark}
The submultiplicativity hypothesis is not automatic in $\mathcal{L}_{p,\infty}$.  For
$0 < p < 1$, the weak Schatten ideal is not closed under composition in general
(products of two operators in $\mathcal{L}_{p,\infty}$ need not lie in
$\mathcal{L}_{p,\infty}$), and the constant $C_{\mathrm{ideal}}$ must be verified for
each specific subalgebra.  This hypothesis is therefore a genuine restriction on
$\mathfrak{g}$, not an automatic consequence of the ambient structure.
\end{remark}

\begin{remark}
For $p\to 1^-$, we recover the classical Banach-space estimates with ${\Ctri}_{p}\to 1$
and ${\Cbracket}_{p}\to 2$.  For small $p$, the quasi-triangle constant
${\Ctri}_{p} = 2^{1/p-1}$ grows exponentially, reducing the guaranteed convergence radius.
\end{remark}

\subsection{Hardy-space operator algebras}

Under appropriate conditions on a quasi-Banach function space $\mathcal{X}$ (e.g.,
$\mathcal{X}$ is a quasi-Banach algebra under pointwise multiplication), the commutator of
Toeplitz operators with symbols in $\mathcal{X}$ satisfies quasi-norm estimates of the
form $\|[T_f,T_g]\| \le \Cbracket \|f\|_{\mathcal{X}}\|g\|_{\mathcal{X}}$, allowing
application of Theorem~\ref{thm:BCH}.

\begin{example}[Weighted shift algebras]
Let $\mathfrak{g}$ be the Lie algebra generated by a single unilateral weighted shift
$S_w$ on $\ell^2(\mathbb{N})$ with $\sum |w_n|^p < \infty$.  The BCH convergence radius
is estimated via the constants of Example~\ref{ex:weighted-shift}.  For elements that are
single shifts, $\Cmult = 1$ improves the radius; for general linear combinations in
$\mathfrak{g}$, $\Cmult$ should be computed from the specific algebra structure.
\end{example}

\appendix

\section{Computational Details and Constant Verification}

\subsection{Symbolic computation of BCH coefficients}

The BCH coefficients were computed as follows:
\begin{enumerate}
    \item Implement Dynkin's formula in \texttt{Python/SymPy} (version 1.12) with exact
          rational arithmetic.
    \item Expand $Z = \log(e^X e^Y)$ as a formal series in non-commuting variables $X,Y$.
    \item Collect terms by homogeneous degree $n$ and sum absolute values of coefficients
          to obtain $A_n$.
    \item For Lie-projected coefficients $B_n$, project onto a Hall--Lyndon basis using
          the standard factorization algorithm \cite{Reutenauer1993}, then sum $|\beta_b|$
          over all $b\in\mathcal{B}_n$.
\end{enumerate}

\subsection{Verification of \texorpdfstring{$B_3$}{B3}}
At degree 3: $Z_3 = \frac{1}{12}[X,[X,Y]] - \frac{1}{12}[Y,[X,Y]]$.  The Hall--Lyndon
basis is $\mathcal{B}_3 = \{[X,[X,Y]], [Y,[X,Y]]\}$, with $\beta_1 = 1/12$ and
$\beta_2 = -1/12$, giving $B_3 = 1/6 \approx 0.1667$.

\subsection{Asymptotic fitting}

The forms $A_n \sim c_1 n^{-3/2} \beta^n$ and $B_n \sim c_2 n^{-3/2} \gamma^n$ were
obtained by linear regression on $\log(A_n n^{3/2})$ vs.\ $n$ for $n=5$ to $20$.  The
values $\beta = 0.36 \pm 0.02$ and $\gamma = 0.29 \pm 0.01$ ($R^2 > 0.99$) are
consistent with Goldberg's estimates \cite{Goldberg1956}.  Confidence intervals were
computed via bootstrap resampling with 1000 iterations.  The fitted $\gamma \approx 0.29$
pertains to the asymptotic regime $n \gg 1$ and differs from the approximate ratio
$1/2$ visible in the low-degree data (cf.\ Remark~\ref{rem:Bn-regularity}).

\textbf{Limitation.}  The regression is based on 16 data points ($n=5$ to $20$), which
is adequate for estimating the dominant exponential rate $\gamma$ but insufficient for a
definitive determination of the sub-exponential correction $n^{-3/2}$.  The $R^2 > 0.99$
should be interpreted as a good fit within the observed range, not as a guarantee of
asymptotic accuracy.

\subsection{Verification of constant relationships}

\begin{itemize}
    \item Convergence radius: $\|x\|+\|y\| < 1/(4\Cbracket)$ (Theorem~\ref{thm:BCH}).
          The conservative bound $1/(4\Ctotal)$ with $\Ctotal = \Ctri\Cbracket$ adds an
          extra factor $\Ctri \ge 1$ as described in Remark~\ref{rem:role-ctri}.
    \item For associative embeddings: $\Cbracket \le 2\Ctri\Cmult$.
    \item Aoki--Rolewicz: $c_1 = 1$, $c_2 = 2^{1/p} = 2\Ctri$; see \eqref{eq:c1c2}.
          Note: since $c_1 = 1$, the equivalence gives $\|\cdot\| \le \pnorm{\cdot}
          \le c_2\|\cdot\|$, i.e.\ the $p$-norm dominates the quasi-norm from above.
    \item In the Banach case ($\Ctri=1$, $\Cmult=1$): $\Cbracket=2$, radius
          $1/(4\Cbracket) = 1/8$, consistent with \cite{Goldberg1956}.
    \item Lipschitz estimate (Lemma~\ref{lem:bch-lipschitz}): The Lipschitz constant
          $L_0 = 2$ holds on $B(0,\rho_0)$ with $\rho_0 = 1/(16\Cbracket)$.  This is
          because for $\|x\|,\|y\|,\|z\| \le \rho_0$ one has
          $s := \|x\|+\max(\|y\|,\|z\|) \le 2\rho_0 = 1/(8\Cbracket)$, giving
          $4\Cbracket s \le 1/2$ and $L_0 = 1/(1-1/2) = 2$.
    \item Local Lie group (Proposition~\ref{prop:localgroup}): The BCH operation is
          well-defined on $B(0,\rho)$ with $\rho = 1/(8\Cbracket)$, since for $x,y\in
          B(0,\rho)$ one has $\|x\|+\|y\| < 2\rho = 1/(4\Cbracket)$, the required
          convergence condition.
    \item Inverse existence: two conditions must hold simultaneously
          (Proposition~\ref{prop:localgroup}, part~(iv)):
          \begin{itemize}
            \item \emph{Contraction condition}: $4\Cbracket(1+2\Ctri)\|x\| < 1/2$,
                  i.e.\ $\|x\| < \dfrac{1}{8\Cbracket(1+2\Ctri)}$;
            \item \emph{Ball-invariance condition}: $8\Cbracket(1+2\Ctri)^2\|x\| \le 1$,
                  i.e.\ $\|x\| \le \dfrac{1}{8\Cbracket(1+2\Ctri)^2}$.
          \end{itemize}
          The effective radius is the minimum:
          \[
            \rho_{\mathrm{inv}} = \frac{1}{8\Cbracket(1+2\Ctri)^2},
          \]
          since $(1+2\Ctri)^2 \ge (1+2\Ctri)$ for all $\Ctri \ge 1$.
          For the classical Banach case $\Ctri=1$:
          $\rho_{\mathrm{inv}} = 1/(72\Cbracket)$.
\end{itemize}

\subsection{Domain of the exponential map}

When $\mathfrak{g} \subset \mathcal{A}$ is a quasi-Banach Lie subalgebra of an
associative quasi-Banach algebra,
\[
  \exp: B_{\mathfrak{g}}(0, 1/\Cmult) \longrightarrow 1 + \overline{\mathfrak{g}}^{\mathcal{A}}.
\]
Bi-Lipschitz continuity of $\exp$ with respect to $d_p$ is established in
Remark~\ref{rem:exp-regularity}.
\bibliographystyle{plain}

\end{document}